\documentclass[a4paper,reqno,11pt]{amsart}

\usepackage[left=1 in, right=1 in,top=1 in, bottom=1 in]{geometry}

\usepackage{amsfonts}
\usepackage{amssymb}
\usepackage{amsthm}
\usepackage{amsmath}
\usepackage{mathrsfs}
\usepackage{color}

\usepackage[numbers,sort&compress]{natbib} 

\usepackage{pdfsync} 

\allowdisplaybreaks

\numberwithin{equation}{section} 

\newcommand*\rmd{\mathop{}\!\mathrm{d}}

\newcommand{\rme}{{\mathrm{e}}}

\newcommand{\supp}{\mathrm{supp}}
\newcommand{\tr}{\mathrm{tr}}

\newcommand{\mbr}{\mathbb{R}}
\newcommand{\mbt}{\mathbb{T}}
\newcommand{\mbs}{\mathbb{S}}
\newcommand{\mbn}{\mathbb{N}}
\newcommand{\mbz}{\mathbb{Z}}
\newcommand{\mbq}{\mathbb{Q}}
\newcommand{\mbc}{\mathbb{C}}
\newcommand{\mbw}{\mathbb{W}}
\newcommand{\mbp}{\mathbb{P}}

\newcommand{\pt}{\partial_t}

\newcommand{\Linfone}[1]{ \Big\| #1 \Big\|_{L_t^\infty L_x^1}}  
\newcommand{\linfone}[1]{ \| #1 \|_{L_t^\infty L_x^1}}			%
\newcommand{\linftwo}[1]{ \| #1 \|_{L_t^\infty L_x^2}}			%
\newcommand{\Linfp}[1]{ \Big\| #1 \Big\|_{L_t^\infty L_x^p}}	
\newcommand{\linfp}[1]{ \| #1 \|_{L_t^\infty L_x^p}}		
\newcommand{\Linfinf}[1]{ \Big\| #1 \Big\|_{L_t^\infty L_x^\infty}}	
\newcommand{\linfinf}[1]{ \| #1 \|_{L_t^\infty L_x^\infty}}

\newcommand{\linftwopsq}[1]{\| #1 \|_{L_t^\infty L_x^{2p}}^2 }

\newcommand{\linftwop}[1]{\| #1 \|_{L_t^\infty L_x^{2p}} }

\newcommand{\CN}[1]{\left\|#1\right\|_{C_{t,x}^N}}
\newcommand{\cN}[1]{\|#1\|_{C_{t,x}^N}}
\newcommand{\cNN}[2]{\|#1\|_{C_{t,x}^{#2}}}
\newcommand{\CNN}[2]{\left\|#1\right\|_{C_{t,x}^{#2}}}
\newcommand{\Cone}[1]{\left\|#1\right\|_{C_{t,x}^1}}

\newcommand{\opR}{\mathcal{R}}

\newcommand{\mce}[1]{\mathcal{E}_{#1}}

\newcommand{\divg}{\nabla\!\!\cdot\!}
\newcommand{\perpdot}{\nabla^\perp\!\!\cdot}

\newcommand{\wqp}{w^{(\mathrm{p})}}
\newcommand{\wqc}{w^{(\mathrm{c})}}
\newcommand{\wqt}{w^{(\mathrm{t})}}

\newcommand{\sumk}{\sum_{\dir \in \Lambda}}
\newcommand{\sumkk}{\sum_{\dir,\dir'\in \Lambda}}

\newcommand{\Rl}{R_{\mathrm{linear}}}
\newcommand{\Rc}{R_{\mathrm{corrector}}}
\newcommand{\Ro}{R_{\mathrm{oscillation}}}
\newcommand{\Rm}{\oR_{\mathrm{commutator}}}

\newcommand{\Nrang}{N =1,2,3}

\newcommand{\ls}{\lesssim}
\newcommand{\lamq}{\lambda_q}
\newcommand{\lamqp}{\lambda_{q+1}}

\newcommand{\oR}{\mathring{R}}
\newcommand{\ootimes}{\mathring{\otimes}}
\newcommand{\Rls}{\oR_\ell^*}
\newcommand{\dir}{\xi}

\newcommand{\id}{I\!d}

\newcommand{\dM}{\mathcal{M}}

\def\Xint#1{\mathchoice
	{\XXint\displaystyle\textstyle{#1}}%
	{\XXint\textstyle\scriptstyle{#1}}%
	{\XXint\scriptstyle\scriptscriptstyle{#1}}%
	{\XXint\scriptscriptstyle\scriptscriptstyle{#1}}%
	\!\int}
\def\XXint#1#2#3{{\setbox0=\hbox{$#1{#2#3}{\int}$ }
		\vcenter{\hbox{$#2#3$ }}\kern-.6\wd0}}
\def\aint{\Xint{\diagup}}

\newtheorem{lem}{Lemma}[section]
\newtheorem{cor}[lem]{Corollary}
\newtheorem{prop}[lem]{Proposition}
\newtheorem{thm}[lem]{Theorem}
\newtheorem{defn}[lem]{Definition}

\begin{document}

\title[ ]{ Non-Uniqueness of Weak Solutions to 2D Hypoviscous Navier-Stokes Equations }

\author{Tianwen Luo}
\address{Yau Mathematical Sciences Center, Tsinghua University, China.}
\email[T. Luo]{twluo@mail.tsinghua.edu.cn}
\thanks{The work of T.L. is supported in part by NSFC Grants 11601258.}

\author{Peng Qu}
\address{
School of Mathematical Sciences \& Shanghai Key Laboratory for Contemporary Applied Mathematics, Fudan University, China}
\email[P. Qu]{pqu@fudan.edu.cn}
\thanks{The work of P.Q. is supported in part by  NSFC Grants 11831011 and 11501121.}




\begin{abstract}
Through an adaption of {the} convex integration scheme in {the} two dimensional case, the non-uniqueness of $C^0_t L^2_x$ weak solutions is presented for {the two-dimensional} hypoviscous incompressible Navier-Stokes equations.
\end{abstract} 

\maketitle


\section{Introduction}\label{sec:1}

In this paper, we consider the 2D incompressible Navier--Stokes equations with fractional viscosity
\begin{equation} \label{1.1_NS}
\left\{ \begin{aligned}
& \pt v + \divg (v \otimes v ) + \nabla p + \nu (-\Delta)^{\theta} v = 0, \\
& \divg v = 0,
\end{aligned} \right.
\end{equation} 
where $\theta \in [0,1)$ is a given constant, the velocity field $v=v(t,x)$ is defined on  $(t,x) \in [0,+\infty) \times \mbt^2$ with zero spatial means
\begin{equation}
\int_{\mbt^2}^{} v(t,x) \rmd x = 0, 
\end{equation}
and we denote $\mbt^2 = \mbr^2 /( 2\pi \mbz^2)$. Here, for $u \in C^{\infty}(\mathbb{T}^3)$ the fractional Laplacian is defined via the Fourier transform as
\begin{align*}
\mathcal{F}((-  \Delta)^{\theta} u)(\xi) = |\xi|^{2\theta}\mathcal{F}(u)(\xi), \quad \xi \in \mathbb{Z}^2.
\end{align*}

When $\theta = 1$, System \eqref{1.1_NS} is the 2D Navier-Stokes equations, for which the existence and uniqueness of weak solutions to the Cauchy problem are well-established (see, for example, \cite{Temam-NSbook}). These weak solutions also satisfy the energy equality. In contrast, recently Buckmaster and Vicol showed the nonuniqueness of weak solutions to 3D Navier--Stokes equations in \cite{Buckmuster_Vicol}. The 3D Navier Stokes equations with fractional viscosity was  first considered by J.-L. Lions in \cite{Lions59} and the existence and uniqueness of weak solutions to the Cauchy problem for $\theta \in [5/4,\infty)$ was showed  in \cite{Lions69}. Moreover, an analogue of the Caffarelli-Kohn-Nirenberg \cite{CKN} result was established in \cite{KatzPavlovic}, showing that the Hausdorff dimension of the singular set, in space and time, is bounded by $5 - 4\theta$ for $\theta \in (1,5/4)$. The existence, uniqueness, regularity and stability of solutions to the 3D Navier--Stokes with fractional viscosity have been studied in \cite{OlsonTiti05,JiuWang14,Wu03,Tao09,Colombo_DeLellis_Massaccesi,Tang_Yu} and references therein.
On the other hand, for $\theta \in [1,5/4)$,  the non-uniqueness of weak solutions to the 3D Navier Stokes equations with fractional viscosity was showed in \cite{Luo_Titi}, extending the results in \cite{Buckmuster_Vicol}; while for $\theta \in (0,1/5)$, the non-uniqueness of Leray weak solutions was showed in \cite{CdLdR18}. 

The framework of convex integration, applicable to fluid dynamics, was introduced by De Lellis and Sz{\'e}kelyhidi in \cite{dLSz1,DeLellis_Szekelyhidi_InvMath} for the Euler equations. Since then, it was developed in the series of work in \cite{Isett12,Buckmaster2013transporting,Buckmaster2014,Isett16,BDSV17}, over the resolution of the flexible part of Onsager's conjecture for the 3D Euler equations; see also \cite{CET94} for the rigidity part. Recently, the method was extended to Navier--Stokes equations in \cite{Buckmuster_Vicol}, by developing a framework of convex integration with intermittence. The ideas in \cite{Buckmuster_Vicol} are further developed to treat transport equations, Boussinesq, and stationary Naiver-Stokes equations in \cite{Modena_Szekelyhidi,Buckmaster_Colombo_Vicol,LuoX,Cheskidov_Luo,LTZ}.

The purpose of this note is to show that, for the 2D hypoviscous Navier-Stokes equations with $\theta \in [0,1)$, the $C^0_t L^2_x$ weak solutions are not unique. As in  \cite{Luo_Titi}, we would like to show a result of $h$-principle type to this system.

\begin{thm} \label{thm:1}
	For any given $\theta \in [0,1)$ and $T \in \mbr_+$,
	if one has a smooth divergence-free vector field $u = u(t,x)$ with zero spatial mean on $[0,T]\times\mbt^2$,
	then for any given $\varepsilon_*>0$, 
	there exists a weak solution $v=v(t,x) \in C^0_t L_x^2$ to equations \eqref{1.1_NS},
	with zero spatial mean, satisfying
	\begin{gather}
	\linfone{v-u} \leq \varepsilon_*, \label{1.3} \\
	\supp_t v \subseteq N_{\varepsilon_*} (\supp_t u). \label{1.4}
	\end{gather}
\end{thm}
\noindent{}Here for weak solutions, we mean solutions in the sense of distribution, and see \eqref{2.9} for $N_\varepsilon(\cdot)$.
Moreover, by choosing $u$ with a compact temporal support,
and $\varepsilon_* > 0$ small enough,
we have

\begin{cor}\label{cor:1}
	System \eqref{1.1_NS} admits nontrivial $C^0_t L_x^2$ weak solutions with compact temporal supports.
	Thus, generally, $C^0_t L_x^2$ weak solutions to the Cauchy problem of \eqref{1.1_NS} are not unique.
\end{cor}

We now make some comments on the analysis in this paper. We shall adapt the 2D stationary flow introduced in \cite{Choffrut_DeLellis_Szekelyhidi} to an intermittent form, inspired by the the intermittent Beltrami flow introduced in \cite{Buckmuster_Vicol} as the basic building block in the intermittent convex integration scheme for 3D Navier--Stokes equations.
Meanwhile, in the two-dimensional case, 
it seems that the method of intermittent jets  introduced in \cite{Buckmaster_Colombo_Vicol} or viscous eddies introduced in \cite{Cheskidov_Luo} can not be applied, 
due to the 3D nature of its Mikado flow structure. Furthermore, we shall use different scaling for the parameters due to the $L^p$ estimates for the 2D Dirichlet kernels.
At last, we would like to compare the result of this note with the one of \cite{LTZ}. 
In \cite{LTZ}, the authors present the 2D intermittent convex integration scheme to show the finite energy weak solutions 
for 2D Boussinesq equations with diffusive temperature.
By taking constant temperature in the solution, \cite{LTZ} can also provide the non-uniqueness result to \eqref{1.1_NS}. 
The new points got in this note may be given as follows.
First, Theorem \ref{thm:1} provides a result of the h-principle type.
Secondly, with Theorem \ref{thm:1}, one can construct solutions with compact temporal supports.

\section{Iteration Lemma} \label{sec:2}

In order to prove the above result in the framework of convex integration, 
one needs an iteration process on the corresponding Navier--Stokes--Reynolds  system
\begin{equation} \label{2.1_NSR}
	\left\{\begin{aligned}
		& \pt v + \divg (v \otimes v) + \nabla p + \nu (-\Delta)^\theta v = \divg \oR, \\
		& \divg v = 0,
	\end{aligned}\right.
\end{equation}
where the Reynolds tensor $\oR$ is a symmetric trace-free $2 \times 2$ matrix.
{Also} we would apply the scheme of intermittent convex integration to add waves with high frequency and strong concentration to cancel the Reynolds tensor $\oR$ gradually.

In order to illustrate our analysis in a clearer manner,
we would use several parameters to denote the different scales in the convex integration process.
First, for $\theta \in [0,1)$ given in the system \eqref{1.1_NS},
we denote
\begin{equation} \label{3.58+3}
\theta_* = \left\{\begin{alignedat}{2}
& 2\theta-1, \quad & & \frac{1}{2} < \theta < 1, \\
& 0, \quad & & 0 \leq \theta \leq \frac{1}{2},
\end{alignedat}\right.
\end{equation}
for which, we can easily check that $\theta_* \in [0,1)$.
Then we shall choose the index parameter $\alpha \in \mbq_+$ accordingly satisfying
\begin{equation} \label{7.16+_alpha}
	\alpha \leq \frac{1-\theta_*}{8} \in \big( 0, \min\{ \frac{1-\theta}{4}, \frac{1}{8} \} \big].
\end{equation}
Now for each $q \in \mbn$, we set
\begin{equation} \label{2.2_lambda_q}
	\lambda_q = A^{(B^q)}
\end{equation}
to denote the principle frequency for the perturbation waves in the convex integration scheme,
and set
\begin{equation}\label{2.4_varepsilon}
	\varepsilon_q = \lambda_q^{-2\beta}
\end{equation}
to denote the amplitude.
Here $B \in \mbn$ would be chosen large enough based on $\alpha$ to satisfy
\begin{equation}\label{++.1}
	B > \frac{320}{\alpha},
\end{equation}
and $\beta \in \mbr_+$ would be chosen small enough accordingly to satisfy
\begin{equation}\label{++.2}
	0 < \beta < \frac{1}{100 B^2}. 
\end{equation}
The parameter $A \in 5 \mbn$ would be chosen at last to be large enough 
to absorb the absolute constants in the inequalities and to satisfy
\begin{equation}\label{++.3}
	A^\alpha \in 5 \mbn.
\end{equation}
We note that under these choices, we have 
\begin{equation}\label{++.4}
	\lambda_q \in 5 \mbn, \quad \lambda^\alpha_q \in 5 \mbn, \quad \forall\, q \in \mbn.
\end{equation}
and
\begin{equation}\label{++.5}
	\varepsilon_{q+1}^{-1} \ll \varepsilon_{q+2}^{-1} = \lambda_q^{2 \beta B^2} \leq \lambda_q^{\frac{1}{50}}.
\end{equation}

In the main parts of this note, we would try to prove this iteration lemma

\begin{lem} \label{Lem:2.1}
	For any given $\theta \in [0,1)$ and $T \in \mbr_+$, 
	if $(v_q, p_q, \oR_q)$ is a smooth solution to \eqref{2.1_NSR} on $[0,T] \times \mbt^2$ with
	\begin{gather} 
		\Cone{v_q} \leq \lamq^4, \label{2.2_a_vqC1_asmp} \\
		\Linfone{\oR_q} \leq A \varepsilon_{q+1}, \label{2.2_b_Rq_l1_asmp}\\
		\Cone{\oR_q} \leq \lamq^{10} \label{2.2_c_RC1_asmp}
	\end{gather}
	and $\aint_{\mbt^2} v_q \rmd x = 0$,
	then there exists a smooth solution $(v_{q+1}, p_{q+1}, \oR_{q+1})$ to \eqref{2.1_NSR} with
	\begin{gather}
		\Cone{v_{q+1}} \leq \lamqp^4, \label{2.3_a_vqC1_est} \\
		\Linfone{\oR_{q+1}} \leq A \varepsilon_{q+2}, \label{2.3_b_Rq_l1_est}\\
		\Cone{\oR_{q+1}} \leq \lamqp^{10} \label{2.3_c_RC1_est}
	\end{gather}
	and 
	\begin{gather}
		\supp_t v_{q+1} \cup \supp_t \oR_{q+1} \subset N_{\varepsilon_{q+1}} (\supp_t v_q \cup \supp_t \oR_q), \label{2.4_suppv} \\
		\linftwo{v_{q+1} - v_q} \leq A \varepsilon_{q+1}^{\frac{1}{2}}, \label{2.5_L2Increase} \\
		\linfone{v_{q+1} - v_q} \leq \varepsilon_{q+1}^{\frac{1}{2}}, \label{2.6_WIncrease} \\
		\aint_{\mbt^2} v_{q+1} \rmd x =0, \label{2.6+}
	\end{gather}
	where for $S \subseteq [0,T]$ we denote
	\begin{equation}\label{2.9}
		N_\varepsilon (S) := \big\{ t \in [0,T] \mid \exists s \in S, \ \text{s.t.} \, |s-t| \leq \varepsilon \big\}.
	\end{equation}
\end{lem}

With this iteration lemma we can prove
Theorem \ref{thm:1} as follows

\begin{proof}[Proof of the main theorem.]
	Take $v_0 = u$ and we shall define $p_0, \oR_0$ for the Navier--Stokes--Reynolds system \eqref{2.1_NSR} as
	\[
		\oR_0 = \opR \big( \pt v_0 + \nu (-\Delta)^\theta v_0 \big) + v_0 \ootimes v_0
	\] 
	and
	\[
		p_0 = - \frac{1}{2} |v_0|^2,
	\]
	where $\opR$ would be defined in details in \eqref{3.47_opR} later,
	$\ootimes$ denotes the trace-free part of the tensor product as
	\[
		f \ootimes g = \begin{pmatrix}
			\frac{1}{2} f_1 g_1 - \frac{1}{2} f_2 g_2 & f_1 g_2 \\
			f_2 g_1 & \frac{1}{2} f_2 g_2 - \frac{1}{2} f_1 g_1
		\end{pmatrix}, \quad \forall\, f,g \in \mbr^2.
	\]
	Then for $A$ large enough 
	one can use Lemma \ref{Lem:2.1} to get the sequence $\{v_q\}$ with estimates \eqref{2.3_a_vqC1_est}--\eqref{2.6+}.
	Therefore, by \eqref{2.5_L2Increase}, one has
	\[
		\sum_{q=0}^{\infty} \linftwo{v_{q+1}- v_q}  < +\infty,
	\]
	which shows the strong convergence of $\{v_q\}$ in $L_t^\infty L_x^2$ to some $v(t,x)$.
	And by \eqref{2.3_b_Rq_l1_est}, this $v(t,x)$ is a weak solution to \eqref{1.1_NS}.
	Meanwhile, by \eqref{2.6_WIncrease} and \eqref{2.4_suppv}, 
	one can get \eqref{1.3}--\eqref{1.4}.
	
	Moreover, using \eqref{2.2_a_vqC1_asmp} and \eqref{2.5_L2Increase}, we have that for each $q_* \in \mbn$,
	\begin{align*}
		\sum_{q=q_*}^{\infty} \| v_{q+1} - v_q \|_{C^0_t H^{\beta'}_x} \ls & \sum_{q=q_*}^{\infty} \linftwo{v_{q+1} - v_q}^{1-\beta'}  \big(\cNN{v_{q+1}}{1}^{\beta'} + \cNN{v_{q}}{1}^{\beta'} \big) \\
		\ls & \sum_{q=q_*}^\infty A^{1-\beta'} \varepsilon_{q+1}^{\frac{1-\beta'}{2}} \lamqp^{4\beta'} \\
		\ls & \sum_{q=q_*}^\infty A^{1-\beta'} \lamqp^{4 \beta' - \beta(1-\beta')}.
	\end{align*}
	For $\beta' < \beta / (4+\beta)$, 
	this shows that $\{ v_q \}$ is a Cauchy sequence in $C_t^0 H_x^{\beta'}$ 
	and thus converges strongly and $v(t,x)$ is a $C_t^0 L_x^2$ function.
	Here we use $a \ls b$ to denote $a \leq Cb$ for some absolute constant $C$ 
	independent of the choice of our parameters $B ,\beta$ and $A$, and would be absorbed by $A$
	if needed.
\end{proof}

For the rest of the paper, we would try to prove Lemma \ref{Lem:2.1}.

\section{Mollification}

In order to deal with the possible loss of derivatives in the analysis,
we first mollify the approximate solutions.
Denote
\[
	\varphi_\ell(x) = \frac{1}{\ell^2} \varphi_1(\frac{x}{\ell}), \quad \tilde \varphi_\ell(t) = \frac{1}{\ell} \tilde\varphi_1(\frac{t}{\ell})
\]
as the standard 2D and 1D Friedrichs mollifier sequences respectively, with
\[
	\supp \varphi_1 \subseteq B_1(0), \quad \supp \tilde\varphi_1 \subseteq (-1,1).
\]
Then for
\begin{equation} \label{+.2_ell}
	\ell = \lambda_q^{-20},
\end{equation}
we can mollify $v_q$ and $R_q$ given in Lemma \ref{Lem:2.1} as
\begin{align}
	v_\ell = & (v_q *_x \varphi_\ell) *_t \tilde\varphi_\ell, \label{2.3_vl}\\
	\oR_\ell = & (\oR_q *_x \varphi_\ell) *_t \tilde\varphi_\ell. \label{2.4_Rl}
\end{align}
Since $(v_q,p_q,\oR_q)$ solves \eqref{2.1_NSR},
we know that $(v_\ell,p_\ell,\oR_\ell)$ solves
\begin{equation} \label{+.5_NSR_vl}
\left\{\begin{aligned}
& \pt v_\ell + \divg (v_\ell \otimes v_\ell) + \nabla p_\ell + \nu (-\Delta)^\theta v_\ell = \divg (\oR_\ell+\Rm), \\
& \divg v_\ell = 0,
\end{aligned}\right.
\end{equation}
where we can choose
\begin{align}
	p_\ell = & (p_q *_x \varphi_\ell) *_t \tilde\varphi_\ell +|v_\ell|^2 - \big(|v_q|^2*_x\varphi_\ell \big) *_t \tilde\varphi_\ell, \\
	\Rm = & (v_\ell \ootimes v_\ell) - \big( (v_q \ootimes v_q) *_x \varphi_\ell \big) *_t \tilde\varphi_\ell.
\end{align}
Using the inductive assumptions \eqref{2.2_a_vqC1_asmp}--\eqref{2.2_c_RC1_asmp}, we have
\begin{gather}
	\CN{v_\ell} \ls \lamq^4 \ell^{-N+1} \ls \ell^{-N}, \quad \forall\, \Nrang, \label{+.8_vl_CN} \\
	\CN{\oR_\ell} \ls \lamq^{10} \ell^{-N+1} \ls \ell^{-N}, \quad \forall\, \Nrang, \label{+.9_Rl_CN} \\
	\Linfone{\oR_\ell} \leq \Linfone{\oR_q} \leq A \varepsilon_{q+1}, \label{+.12_Rl_l1} \\
	\linftwo{v_\ell - v_q} + \linfone{v_\ell - v_q} \ls \| v_\ell - v_q \|_{L^\infty_t L^\infty_x }  \ls \ell \Cone{v_q} \ls \lamq^{-16}{.} \label{+.11} 
\end{gather}
Moreover,
\begin{align*}
	\Linfinf{\Rm} \ls & \ell \Cone{v_\ell\ootimes v_\ell} \ls \ell \lamq^8,\\
	\CN{\Rm} \ls &  \ell^{-N+1} \Cone{v_\ell\ootimes v_\ell} \ls \ell^{-N+1} \lamq^8.\\
\end{align*}
Thus, for 
\[
	\Rls \overset{\mathrm{def.}}{=} \oR_\ell + \Rm,
\]
we have
\begin{gather}
	\Linfone{\Rls} \leq A \varepsilon_{q+1} + \ell \lamq^8 \leq 2 A \varepsilon_{q+1}, \label{+.16_Rl_l1} \\
	\CN{\Rls} \ls \ell^{-N} + \ell^{-N+1} \lamq^8 \ls \ell^{-N}, \quad \forall\, \Nrang. \label{+.17_Rl_CN}
\end{gather}
Here we use the fact that by our choice of the parameters \eqref{++.5} and \eqref{+.2_ell}, it holds
\[
	\ell \lambda_q^8 \leq \varepsilon_{q+1}{.}
\]

\section{2D Intermittent Stationary Flow}

In this section, we shall choose the sequence of waves with high frequency and strong concentration to perturb the system and construct $v_{q+1}$.
As presented in \cite{Buckmuster_Vicol}, 
the intermittent Beltrami flow is the basic building block in the intermittent convex integration scheme to prove the nonuniqueness of weak solutions to 3D Navier--Stokes equations.
Meanwhile, in the two-dimensional case, 
it seems that the method of intermittent jets  introduced in {\cite{Buckmaster_Colombo_Vicol} or viscous eddies introduced in \cite{Cheskidov_Luo}} can not be applied, 
due to the 3D nature of its Mikado flow structure.
Now we shall adapt the 2D stationary flow introduced in \cite{Choffrut_DeLellis_Szekelyhidi} to an intermittent form.

First, we specifically choose
\begin{align*}
\Lambda^+ & = \{ \frac{1}{5} (3e_1 \pm 4 e_2), \frac15 (4e_1\pm 3e_2) \}, \\
\Lambda^- & = \{ \frac{1}{5} (-3e_1 \mp 4 e_2), \frac15 (-4e_1\mp 3e_2) \}, 
\end{align*}
{and denote}
\begin{equation}
\Lambda = \Lambda^+ \cup \Lambda^-{.} \label{3.10_Lam}	
\end{equation}
{Then}
\[
\Lambda \subset \mbs^1 \cap \mbq^2, \quad 5 \Lambda \subset \mbz^2
\]
and 
\begin{equation*} 
\min_{\substack{\dir,\dir'\in \Lambda \\ \dir \neq -\dir'}} |\dir+\dir'| \geq \frac{\sqrt{2}}{5}.
\end{equation*}

Now for each $\dir \in \Lambda$ and any frequency parameter $\lambda \in  \mbz^+ \cap 5\mbz$,
we may denote the 2D stationary flow $b_\dir$ and its potential $\psi_\dir$ as
\begin{equation} \label{3.1_def_bpsi}
	b_\dir (x) = b_{\dir,\lambda}(x) := i \dir^\perp \rme^{i\lambda \dir \cdot x} \quad \text{and} \quad  \psi_\dir (x) = \psi_{\dir,\lambda}(x) := \frac{1}{\lambda} \rme^{i\lambda \dir \cdot x}.
\end{equation}
It is easy to check that
\begin{equation} \label{3.2_prop_bpsi}
	b_{\dir,\lambda}(x) = \nabla^\perp \psi_{\dir,\lambda}(x), \quad \divg b_\dir(x) = 0, \quad \perpdot b_{\dir,\lambda}(x) = \Delta \psi_{\dir,\lambda}(x) = - \lambda^2 \psi_{\dir,\lambda}(x),
\end{equation}
\begin{equation} \label{3.2+_conj_bpsi}
	\overline{b_{\dir,\lambda} (x)} = b_{-\dir,\lambda}(x), \quad \overline{\psi_{\dir,\lambda} (x)} = \psi_{-\dir,\lambda}(x),
\end{equation}
and 
\begin{equation} \label{3.3+_est_bpsi}
	\|{b_{\dir,\lambda}}\|_{C^N} \leq \lambda^N, \quad \|{\psi_{\dir,\lambda}}\|_{C^N} \leq \lambda^{N-1}, \quad \forall\, N \in \mbn{,}
\end{equation}
where
\begin{equation*} 
	\dir^\perp = \begin{pmatrix}
	 -k_2 \\ k_1
	\end{pmatrix}, \quad \nabla^\perp = \begin{pmatrix}
		-\partial_{x_2} \\ \partial_{x_1}
	\end{pmatrix}.
\end{equation*}

Moreover, we have
\begin{lem}[Geometric lemma] \label{Lem:3.2}
	Denote $\dM$ as the linear space of $2 \times 2$ symmetric trace-free matrices.
	There exists a set of positive smooth functions $\{ \gamma_\dir \in C^\infty( \dM) \mid \dir \in \Lambda\} $,
	such that for each $\oR \in \dM$, 
	\begin{gather}
		\gamma_{-\dir}(\oR) = \gamma_\dir(\oR), \label{3.13} \\
		\oR = \sum_{\dir\in \Lambda} (\gamma_\dir(\oR))^2  (\dir \ootimes \dir), \label{3.14} \\
		\intertext{and}
		\gamma_\dir (\oR) \ls (1 + |\oR|)^{\frac{1}{2}}. \label{3.14+}
	\end{gather}
\end{lem}
The proof of this lemma is direct, one may check Appendix A for the details.

Now as in \cite{Buckmuster_Vicol}, in order to define the intermittent flow we first present the 2D Dirichlet kernel
\begin{equation} \label{3.15_Dr}
	D_r(x) = \frac{1}{2r+1} \sum_{k \in \Omega_r} \rme^{i k \cdot x} \ \in C^\infty(\mbt^2)  
\end{equation}
with $r \in \mbz^+$ and 
\[
	\Omega_r = \{ k=(k_1, k_2)^T \mid k_i \in \mbz, -r \leq k_i \leq r  \}.
\]
By a direct calculation, it holds that for $1 < p \leq \infty$, 
\begin{equation} \label{3.16_DrEst}
	\| D_r \|_{L^p} \lesssim r^{1-\frac{2}{p}}, \quad \| D_r \|_{L^2} = 2\pi.
\end{equation}
We shall note that these $L^p$ estimates are different from the ones in 3D case as in \cite{Buckmuster_Vicol}{,}
and this dimensional dependence is partially the reason for which we shall use different scaling for our parameters to be chosen later.
Now we can define the {directed-rescaled} Dirichlet kernel with a temporal shift as
\begin{equation} \label{3.17_eata_def}
	\eta_\dir(t,x) = \eta_{\dir,\lambda,\sigma,r,\mu}(t,x) := 
	\left\{ \begin{alignedat}{2}
	& D_r (\lambda \sigma (\dir\cdot x + \mu t), \lambda \sigma \dir^\perp \cdot x), & \quad \dir & \in \Lambda^+, \\
	& \eta_{-\dir,\lambda,\sigma,r,\mu} (t,x), & \quad \dir & \in \Lambda^-
	\end{alignedat} \right.
\end{equation}
with 
\begin{equation} \label{3.18_eta_est}
	\frac{1}{\mu} \pt \eta_\dir(t,x) = \pm (\dir\cdot \nabla) \eta_\dir(t,x), \quad \forall\, \dir \in \Lambda^\pm 
\end{equation}
and 
\begin{equation} \label{3.19_eta_norm}
	\aint_{\mbt^2} \eta_\dir^2 (t,x) \rmd x = 1, \quad \linfp{\eta_\dir} \lesssim r^{1-\frac2p}, \quad \text{for } 1 < p \leq \infty. 
\end{equation}
Here we use parameters $r, \mu, \sigma^{-1}, \lambda \in \mbn$ with
\begin{equation} \label{3.19+_parameters}
	1 \ll r \ll \mu \ll \sigma^{-1}  \ll \lambda 
\end{equation}
and
\[
	\lambda \sigma \in 5 \mbn,
\]
one may check \eqref{3.64_parameter_choice} to see the specific choice of these parameters.
We shall note that the choice of these parameters, especially that of $\mu$, 
are dimensionally dependent and thus are
different from that of \cite{Buckmuster_Vicol}.

Finally, we could define the intermittent 2D stationary flow as
\begin{equation} \label{3.20_W_def}
	\mbw_\dir(t,x) = \mbw_{\dir,\lambda,\sigma,r,\mu} (t,x)  := \eta_{\dir,\lambda,\sigma,r,\mu} (t,x) b_{\dir,\lambda}(x).
\end{equation}
Similar as the 3D intermittent Beltrami flow presented in \cite{Buckmuster_Vicol},
this intermittent flow possesses several important properties.
First, for the frequency projector $\mbp_{[\lambda_1,\lambda_2]}$:
\[
	\mbp_{[\lambda_1,\lambda_2]} f(x) = \mathcal{F}^{-1} (1_{ \{ \lambda_1 \leq \dir \leq \lambda_2\} } \mathcal{F}(f) )(x),
\]
where $\mathcal{F}$ is the Fourier transform on $\mbt^2$,
and for
\begin{align*}
	\mbp_{\geq \lambda} f & = \mbp_{[\lambda,\infty)} f, \\
	\mbp_{\neq 0} f & = f - \aint f \rmd x,
\end{align*}
one has
\begin{align}
	\mbp_{[{\lambda}/{2}, 2 \lambda]} \mbw_{\dir,\lambda}(t,x) & = \mbw_{\dir,\lambda}, \label{3.21}\\
	\mbp_{[ \lambda/5, 4 \lambda ]} \big( \mbw_{\dir,\lambda} \ootimes \mbw_{\dir',\lambda} \big) & = \mbw_{\dir,\lambda} \ootimes \mbw_{\dir',\lambda}, \quad \forall\, \dir + \dir' \neq 0, \label{3.21+} \\
	\mbp_{\geq (\lambda\sigma)/2} \big(\mbw_{\dir,\lambda} \ootimes \mbw_{\dir',\lambda} \big) & = \mbp_{\neq 0} \big(\mbw_{\dir,\lambda} \ootimes \mbw_{\dir',\lambda} \big), \quad \forall\, \dir,\dir' \in \Lambda. \label{3.26+}
\end{align}
Similarly,
\begin{equation} \label{3.20+_eta_FrePro}
\mbp_{\neq 0} \eta_\dir = \mbp_{\geq {(\lambda\sigma)}/{2}} \eta_\dir.
\end{equation}
Next, one can get
\begin{lem} \label{Lem:3.3}
	For any $\{ a_\dir \mid \dir \in \Lambda \} \subset \mbc$ with $a_{-\dir} = \overline{a_\dir}$, the function
	\begin{equation} \label{3.22_Wdef}
		W(t,x) = \sum_{\dir \in \Lambda} a_\dir \mbw_\dir(t,x) 
	\end{equation}
	is real valued, and for each $\oR \in \dM$, one has 
	\begin{equation} \label{3.23}
		\sum_{\dir \in \Lambda} (\gamma_\dir(\oR))^2 \aint_{\mbt^2} \mbw_\dir \ootimes \mbw_{-\dir} \rmd x= - \oR.
	\end{equation}
\end{lem}
\begin{proof}
	This result can be checked directly as follows.
	By \eqref{3.2+_conj_bpsi} and \eqref{3.17_eata_def},
	\[
		\overline{W(t,x)} =  \sum_{\dir \in \Lambda} \overline{a_\dir} \overline{\mbw_\dir(t,x)} 
		= \sum_{\dir \in \Lambda} a_{-\dir} \overline{\eta_{\dir}(t,x)} \overline{b_\dir(t,x)} 
		= \sum_{\dir \in \Lambda} a_{-\dir} {\eta_{-\dir}(t,x)} {b_{-\dir}(t,x)} 
		= W(t,x),
	\]
	and
	\[
		\mbw_\dir \ootimes \mbw_{-\dir} = \eta_\dir^2(t,x) \big( b_\dir(x) \ootimes b_{-\dir}(x) \big) 
		=  \eta_\dir^2 (t,x) \big( \dir^\perp \ootimes \dir^\perp \big) 
		=   \eta_\dir^2 (t,x) ( - \dir \ootimes \dir).
	\]
	Then by \eqref{3.14} and \eqref{3.19_eta_norm}, one can get \eqref{3.23}.
\end{proof}

Moreover, after a direct calculation, one can get
\begin{lem} \label{Lem:3.4}
	If one chooses the parameters as in \eqref{3.19+_parameters}, 
	then for any $1 < p \leq \infty,$ and $K, \Nrang$, one has
	\begin{align}
		\linfp{ \mbw_\dir}  + \linfp{ \nabla^N \pt^K \mbw_\dir} \lesssim &\lambda^N \big( \lambda \sigma r \mu \big)^K \ r^{1 - \frac2p}, \label{3.24_mbw_est} \\
		\linfp{ \eta_\dir}  + \linfp{ \nabla^N \pt^K \eta_\dir} \lesssim &\big(\lambda \sigma r \big)^N \big( \lambda \sigma r \mu \big)^K \ r^{1 - \frac2p}. \label{3.25_eta_est} 		
	\end{align}
\end{lem}

\section{Perturbation}

To present our perturbation terms, 
we first define the temporal cutoff as in \cite{Luo_Titi}.
Let $\Phi_q(t)$ be a smooth cut-off function with
\begin{gather*}
	0 \leq \Phi_q \leq 1, \\
	\Phi_q(t) = 1 \quad \text{on} \ \supp_t \Rls, \\
	\supp \Phi_q(t) \subseteq N_{\ell} (\supp_t \Rls), \\
	\| \Phi_q \|_{C^N_t} \ls \ell^{-N}, \quad \forall\, \Nrang.
\end{gather*}
Then we can set the smooth coefficients
\begin{equation} \label{3.28+_ak_def}
	a_\dir(t,x) = A^{\frac{1}{2}} \varepsilon_{q+1}^{\frac{1}{2}} \gamma_\dir (A^{-1} \varepsilon_{q+1}^{-1} \Rls(t,x)) \Phi_q(t),
\end{equation}
for $\dir \in \Lambda$.
Obviously,
\begin{equation} \label{3.27}
	\supp_t a_\dir \subseteq N_{\ell} (\supp_t \Rls),
\end{equation}
and by \eqref{3.23}, it is easy to see that
\begin{equation} \label{3.34}
\sumk a_\dir^2 \aint \mbw_\dir \ootimes \mbw_{-\dir} \rmd x = - \Rls,
\end{equation}
namely, noting \eqref{3.21+},
\begin{equation} \label{3.42+}
- \Rls = \sumkk a_\dir a_{\dir'} \mbp_{=0} \big( \mbw_\dir \ootimes \mbw_{\dir'} \big).
\end{equation}

Now we can define the perturbation 
\begin{equation} \label{3.29_wq_def}
	w_{q+1} = v_{q+1} - v_\ell := \wqp + \wqc + \wqt,
\end{equation}
where
\begin{align}
	\wqp (t,x) = & \sumk a_\dir(t,x) \mbw_{\dir,\lamqp}(t,x) = \sumk a_\dir(t,x) \eta_{\dir,\lamqp,\sigma,r,\mu}(t,x) b_{\dir,\lamqp}(x), \label{3.30_wqp_def}\\
	\wqc (t,x) = & \sumk \nabla^\perp \big( a_\dir(t,x) \eta_{\dir,\lamqp,\sigma,r,\mu}(t,x) \big) \psi_{\dir,\lamqp}(x), \label{3.31_wqc_def} \\
	\wqt (t,x) = & \frac{1}{\mu} \Big( \sum_{\dir \in \Lambda^+} - \sum_{\dir \in \Lambda^-}  \Big) \mbp_H \mbp_{\neq 0} \big( a_\dir^2(t,x) \mbp_{\neq 0}\eta_{\dir,\lamqp,\sigma,r,\mu}^2(t,x) \dir \big) . \label{3.32_wqt_def}
\end{align}
{Here} $\mbp_H$ is the Helmholtz--Leray projector
\[
	\mbp_H f = f - \nabla \big( \Delta^{-1} \divg f \big).
\]
Moreover, it is direct to check that
\begin{gather}
	\wqp + \wqc = \nabla^\perp \Big( \sumk a_\dir \eta_\dir \psi_\dir \Big), \label{3.33}\\
	\divg (\wqp + \wqc) = 0, \quad \divg \wqt = 0, \label{3.33+} \\
	\supp_t w_{q+1} \subseteq \bigcup_{\dir \in \Lambda} \supp_t a_\dir \subseteq N_{\ell} (\supp_t \Rls). \label{3.39+}
\end{gather}

\section{A Priori Estimates for the Perturbations}

In this section, we derive a priori estimates for the perturbations given above.

\begin{lem}[Estimates for the coefficients] \label{Lem:3.5}
	For $a_\dir$ defined in \eqref{3.28+_ak_def}, one has
	\begin{align}
		\linftwo{a_\dir}  \ls & A^{\frac12} \varepsilon_{q+1}^{\frac12}, \label{3.38_ak_est} \\
		\cN{a_\dir} \ls & \ell^{-2N}, \quad \forall\, \Nrang. \label{3.39}
	\end{align}
\end{lem}
\begin{proof}
%
	
	By Lemma \ref{Lem:3.2} and \eqref{+.16_Rl_l1}--\eqref{+.17_Rl_CN}, we have
	\[
		\linftwo{a_\dir}^2 \ls \int_{\mbt^2} A \varepsilon_{q+1} \cdot \Big( 1 + \frac{|\Rls(t,x)|}{A \varepsilon_{q+1}} \Big) \rmd x \ls A \varepsilon_{q+1} + \Linfone{\Rls} \ls A \varepsilon_{q+1}
	\]
	and
	\begin{align*}
		\cN{a_\dir} \ls & A^{\frac{1}{2}} \varepsilon_{q+1}^{\frac{1}{2}} \| \Phi_q \|_{C^N_t} + A^{\frac{1}{2}} \varepsilon_{q+1}^{\frac{1}{2}} \cN{\gamma_\dir (A^{-1} \varepsilon_{q+1}^{-1} \Rls)} \\
		\ls & A^{\frac{1}{2}} \varepsilon_{q+1}^{\frac{1}{2}} \ell^{-N} + A^{\frac{1}{2}} \varepsilon_{q+1}^{\frac{1}{2}}  \cdot (A^{-1} \varepsilon_{q+1}^{-1} )^{N} \ell^{-N}\\
		\ls & \ell^{-2N},
	\end{align*}
	which leads to \eqref{3.38_ak_est}--\eqref{3.39}.
\end{proof}

Now we present an important tool introduced in \cite{Buckmuster_Vicol}, see also \cite{Modena_Szekelyhidi}.
\begin{lem}[$L^p$ product estimate] \label{Lem:3.6}
	If $f, g \in C^\infty(\mbt^2)$,
	and $g$ is $(\mbt / \kappa)^2$ periodic for some $\kappa \in \mbz^+$, then 
	\begin{equation} \label{3.40}
		\| fg \|_{L^2(\mbt^2)} \leq \| f \|_{L^2(\mbt^2)} \| g \|_{L^2(\mbt^2)} + C \kappa^{-\frac{1}{2}} \| f \|_{C^1(\mbt^2)} \| g \|_{L^2(\mbt^2)}. 
	\end{equation}
\end{lem}
\begin{proof}
	See Lemma 2.1 of \cite{Modena_Szekelyhidi}, and also Lemma 3.6 of \cite{Buckmuster_Vicol}.
\end{proof}

%
Then we can derive the estimates on the perturbations as follows.

\begin{prop}
	If one chooses the parameters as in \eqref{3.19+_parameters},
	then for $1 < p \leq \infty$ and $\Nrang$, one has
	\begin{gather}
		\linftwo{\wqp} \lesssim A^{\frac{1}{2}} \varepsilon_{q+1}^{\frac12} + \ell^{-2} (\lamqp \sigma)^{-\frac12}, \label{3.42_wqp_inf2} \\
		\linfp{\wqc} + \linfp{\wqt} \lesssim \ell^{-4} \big( \sigma  + \mu^{-1}  \big) r^{2-\frac{2}{p}}, \label{3.44_wqcwqt_infp} \\
		\linfp{\wqp} + \linfp{w_{q+1}} \lesssim \ell^{-4} r^{1-\frac{2}{p}}, \label{3.43_w_infp} \\
		\linfp{\pt \wqp} + \linfp{\pt \wqc} \lesssim \ell^{-4} \lamqp \sigma \mu r^{2-\frac{2}{p}}, \label{3.45_pt_w} \\
		\linfp{\nabla^N \wqp} + \linfp{\nabla^N \wqc} + \linfp{\nabla^N \wqt} \lesssim \ell^{-4N} r^{1-\frac{2}{p}} \lamqp^N. \label{3.46_w_cN}
	\end{gather}
\end{prop}
\begin{proof}
	Due to \eqref{3.17_eata_def} and \eqref{3.20_W_def},
	$\mbw_\dir(t,\cdot)$ is $\big(\mbt / (\lambda \sigma)\big)^2$-periodic.
	Thus noting the definition \eqref{3.30_wqp_def} of $\wqp$, 
	and applying Lemma \ref{Lem:3.6}, one can get
	\[
		\linftwo{\wqp} \lesssim \linftwo{a_\dir} \linftwo{\mbw_\dir} + (\lamqp \sigma)^{-\frac12} \| a_\dir \|_{C^1_{t,x}} \linftwo{\mbw_\dir},
	\]
	which, by \eqref{3.24_mbw_est}, \eqref{3.38_ak_est}--\eqref{3.39}, 
	leads to \eqref{3.42_wqp_inf2}.
	Meanwhile, by \eqref{3.24_mbw_est}, 
	\[
		\linfp{\wqp} \lesssim \| a_\dir \|_{C^0_{t,x}} \linfp{\mbw_\dir} \lesssim \ell^{-2} r^{1-\frac{2}{p}}.
	\]
	Noting furthermore the definitions \eqref{3.31_wqc_def}--\eqref{3.32_wqt_def} of $\wqc$ and $\wqt$, 
	and using \eqref{3.3+_est_bpsi}, \eqref{3.25_eta_est},
	\begin{align*}
		\linfp{\wqc}  \lesssim & \| \psi_\dir \|_{L^\infty_x} \| {a_\dir} \|_{C^1_{t,x}} \big( \linfp{\eta_\dir} + \linfp{\nabla \eta_\dir} \big) \\
		\lesssim & \ell^{-2} \big( \sigma r + \lamqp^{-1} \big) r^{1-\frac{2}{p}} \lesssim \ell^{-2} \sigma  r^{2-\frac{2}{p}},\\
		\linfp{\wqt} \lesssim & \frac{1}{\mu} \| a_\dir \|_{C^0_{t,x}}^2 \linftwopsq{\eta_\dir} \\
		\lesssim & \ell^{-4} \frac{1}{\mu} r^{2-\frac{2}{p}},
	\end{align*}
	which yields \eqref{3.44_wqcwqt_infp}--\eqref{3.43_w_infp}.
	
	Similarly, by \eqref{3.24_mbw_est}--\eqref{3.25_eta_est} of Lemma \ref{Lem:3.4},
	\begin{align*}
		\linfp{\pt \wqp} \lesssim & \cNN{a_\dir}{1} \big( \linfp{\mbw_\dir} + \linfp{\pt \mbw_\dir} \big) \\
		\lesssim & \ell^{-2}  \lamqp \sigma r \mu \, r^{1-\frac{2}{p}} \\
		\linfp{\pt \wqc} \lesssim & \| \psi_\dir \|_{L^\infty_{t,x}} \cNN{a_\dir}{2} \big( \linfp{\eta_\dir} + \linfp{\pt \nabla \eta_\dir} \big) \\
		\lesssim & \frac{1}{\lambda} \ell^{-4}  \lamqp^2 \sigma^2 r^2 \mu \, r^{1-\frac{2}{p}},
	\end{align*}
	which leads to \eqref{3.45_pt_w}.
	
	At last, 
	\begin{align*}
		\linfp{\nabla^N \wqp} \lesssim &  \cN{a_\dir} \big( \linfp{\mbw_\dir} + \linfp{\nabla^N \mbw_\dir} \big) \\
		\lesssim & \ell^{-2N} \lamqp^N r^{1-\frac{2}{p}}, \\
		\linfp{\nabla^N \wqc} \lesssim & \cNN{a_\dir}{N+1} \big( \cN{\psi_\dir} (\linfp{\eta_\dir} + \linfp{\nabla \eta_\dir}) + \cNN{\psi_\dir}{0} \linfp{\nabla^{N+1} \eta_\dir} \big) \\
		\lesssim & \ell^{-2N-2}  \lamqp^N r^{1-\frac{2}{p}}, \\
		\linfp{\nabla^N \wqt} \lesssim & \frac{1}{\mu} \cN{a_\dir^2} \big( \linftwop{\eta_\dir} \linftwop{\nabla^N \eta_\dir} \big) \\
		\lesssim & \ell^{-4N} \frac{1}{\mu} \big(\lamqp \sigma r\big)^N r^{2-\frac{2}{p}} \\
		\lesssim & \ell^{-4N} \lamqp^N r^{1-\frac{2}{p}},
	\end{align*}
	which yields \eqref{3.46_w_cN}.
\end{proof}

\section{Anti-divergence Operator and Estimates on the Reynolds Stress Tensor}

As in \cite{DeLellis_Szekelyhidi_InvMath} and \cite{Choffrut_DeLellis_Szekelyhidi}, we shall define the anti-divergence operator $\opR$ as 
\begin{defn}
	For $f \in C^0(\mbt^2,\mbr^2)$, set
	\begin{equation} \label{3.47_opR}
		\opR f = \nabla g + (\nabla g)^T - (\divg g) \id,
	\end{equation}
	where $g$ satisfies
	\[
		\Delta g = f - \aint_{\mbt^2} f \rmd x \quad \text{and} \quad \aint_{\mbt^2} g = 0.
	\]
\end{defn}

\begin{lem}[Lemma 10 of \cite{Choffrut_DeLellis_Szekelyhidi}, Properties of the anti-divergence operator] \label{Lem:3.9}
	For any $f \in \allowbreak C^0(\mbt^2, \allowbreak \mbr^2)$ with $\aint_{\mbt^2} f \rmd x =0$, 
	one has 
	\[
		(\opR f(x))^T = \opR f(x), \quad \tr(\opR f(x)) = 0, \quad \forall\, x \in \mbt^2
	\]
	and
	\[
		\divg \opR f = f, \quad \aint_{\mbt^2} \opR f(x) \rmd x = 0.
	\]
\end{lem}

Moreover, with standard Calderon--Zygmund estimates and Schauder estimates, one can get
\begin{lem} \label{Lem:3.10}
	For $1 < p < \infty$,
	\begin{gather}
		\| \opR \|_{L^p \to W^{1,p}} \lesssim 1, \quad \| \opR \|_{C^0 \to C^0} \lesssim 1, \label{3.49} \\
		\| \opR \mbp_{\neq 0} v \|_{L^p} \lesssim \big\| |\nabla|^{-1} \mbp_{\neq 0} v  \big\|_{L^p}. \label{3.50}
	\end{gather}
\end{lem}

{And} we could use the following lemma to gain a $\lambda^{-1}$ weight when we apply $\opR$ on certain terms.
\begin{lem} \label{Lem:3.11}
	For any given $1 < p < \infty$, $\lambda \in \mbz^+$, $a \in C^2(\mbt^2,\mbr)$ and $f \in L^p(\mbt^2,\mbr^2)$, one has
	\[
		\big\| |\nabla|^{-1} \mbp_{\neq 0} (a \mbp_{\geq \lambda} f ) \big\|_{L^p} \lesssim \lambda^{-1} \|a\|_{C^2} \|f\|_{L^p}.
	\]
\end{lem}
\begin{proof}
	See Lemma B.1 of \cite{Buckmuster_Vicol}. 
	In fact,
	\begin{align*}
		\big\| |\nabla|^{-1} \mbp_{\neq 0} (a \mbp_{\geq \lambda} f) \big\|_{L^p} \leq & 
		\big\| |\nabla|^{-1} \mbp_{\geq {\lambda}/{3}} \big( (\mbp_{\leq {\lambda}/{2}} a) (\mbp_{\geq \lambda} f) \big) \big\|_{L^p} 
		+ \big\| |\nabla|^{-1} \mbp_{\neq 0}  \big( (\mbp_{\geq {\lambda}/{2}} a) (\mbp_{\geq \lambda} f) \big) \big\|_{L^p} 		\\
		\lesssim & \lambda^{-1} \| (\mbp_{\leq {\lambda}/{2}} a) (\mbp_{\geq \lambda} f)\|_{L^p} + \| (\mbp_{\geq {\lambda}/{2}} a) (\mbp_{\geq \lambda} f) \|_{L^p} \\
		\lesssim & \lambda^{-1} \| a \|_{L^\infty} \| \mbp_{\geq {\lambda}} f \|_{L^p} + \| \mbp_{\geq {\lambda}/{2}} a \|_{L^\infty} \| \mbp_{\geq {\lambda}} f \|_{L^p} \\
		\lesssim & \lambda^{-1} \big( \|a\|_{L^\infty} + \lambda \| \mbp_{\geq {\lambda}/{2}} a \|_{W^{1,2+}} \big) \| \mbp_{\geq \lambda} f \|_{L^p} \\
		\lesssim & \lambda^{-1} \big( \|a\|_{L^\infty} + \| \nabla \mbp_{\geq {\lambda}/{2}} a \|_{W^{1,2+}} \big) \| \mbp_{\geq \lambda} f \|_{L^p} \\
		\lesssim & \lambda^{-1} \big( \| a \|_{L^\infty} + \| \nabla^2 a \|_{L^\infty} \big) \| f\|_{L^p}. \qedhere
	\end{align*}
\end{proof}

Now we shall settle an expression formula for $\oR_{q+1}$.
In fact, noting that both $(v_\ell, p_\ell, \Rls)$ and $(v_{q+1}, p_{q+1}, \oR_{q+1})$ solve \eqref{2.1_NSR},
and using the definitions \eqref{3.29_wq_def}--\eqref{3.32_wqt_def}, 
one can get
\begin{align*}
	\divg \oR_{q+1} = & \pt v_{q+1} + \divg (v_{q+1} \ootimes v_{q+1}) + \nabla p_{q+1} + \nu (-\Delta)^\theta v_{q+1} \\
	= & \big( \pt v_\ell + \divg (v_\ell \ootimes v_\ell) + \nabla p_\ell + \nu (-\Delta)^\theta v_\ell - \divg \Rls \big) \\
	& + \pt \big(  \wqp + \wqc + \wqt \big) + \divg \big( v_\ell \ootimes w_{q+1} + w_{q+1} \ootimes v_\ell \big)\\
	& + \divg \big( \wqp \ootimes \wqp + (\wqc + \wqt) \ootimes w_{q+1} + \wqp \ootimes (\wqc + \wqt) \big) \\
	& + \nabla (p_{q+1} - p_\ell) + \nu (-\Delta)^\theta w_{q+1} + \divg \Rls.
\end{align*}
Thus, as in \cite{Buckmuster_Vicol}, if we denote
\begin{align}
	\Rl = & \opR \Big( \pt \wqp + \pt \wqc + \nu (-\Delta)^\theta w_{q+1} \Big) +  v_\ell \ootimes w_{q+1} + w_{q+1} \ootimes v_\ell,	\label{3.51_Rl} \\
	\Rc = & \Big( (\wqc+\wqt) \ootimes w_{q+1} + \wqp \ootimes (\wqc + \wqt) \Big), \label{3.52_Rc} \\
	\Ro = & \wqp \ootimes \wqp + \Rls + \pt \opR \wqt,
\end{align}
we can choose
\begin{equation} \label{3.54_R}
	\oR_{q+1} = \opR \divg \big( \Rl + \Rc + (\Ro - p^* \id) + (p_{q+1} - p_\ell + p^*) \id \big)
\end{equation}
for $p^*$ to be chosen later.
Then obviously, if we properly choose $p_{q+1}$, we have
\begin{equation} \label{3.58+1}
	\supp_t \oR_{q+1} \subseteq \supp_t w_{q+1} \cup \supp_t \Rls \subseteq N_{2\ell} (\supp_t R_q).
\end{equation}

For $\Rc$, by \eqref{3.44_wqcwqt_infp}--\eqref{3.43_w_infp} and \eqref{3.46_w_cN}, we have
\begin{align}
	\linfp{\opR \divg \Rc} 
	\lesssim & \linfp{\Rc} \notag \\
	\lesssim & \big( \linftwop{\wqc} + \linftwop{\wqt} \big) \cdot (\linftwop{w_{q+1}} + \linftwop{\wqp} )\notag \\
	\lesssim & \ell^{-8} \big( \sigma r + \mu^{-1} r \big) r^{2-\frac{2}{p}} \label{3.58+2}
\end{align}
and
\begin{align}
	& \CNN{\opR \divg \Rc}{1} \notag \\
	\ls & (\cNN{\wqp}{2} + \cNN{\wqc}{2} + \cNN{\wqt}{2}) \cdot (\linfinf{\wqp} + \linfinf{\wqc} + \linfinf{\wqt}) \notag \\
	\ls & \ell^{-8} r \lamqp^2 \cdot \ell^{-4} r \ls \ell^{-12} r^2 \lamqp^2. \label{7.9+}
\end{align}

Meanwhile, for $\Rl$ by \eqref{3.33} and \eqref{3.3+_est_bpsi}, \eqref{3.25_eta_est}, \eqref{3.39}, it holds that
\begin{align*}
	\linfp{\opR (\pt \wqp + \pt \wqc) } = & \Linfp{\opR \pt \nabla^\perp \big( \sumk a_\dir \eta_\dir \psi_\dir \big)} \\
	\lesssim & \Linfp{ \sumk \pt (a_\dir \eta_\dir) \psi_\dir} \\
	\lesssim & \ell^{-2} \sigma \mu r^{2-\frac{2}{p}}.
\end{align*}
By \eqref{3.43_w_infp}, \eqref{3.46_w_cN} and \eqref{+.8_vl_CN},
\begin{gather*}
	\linfp{ \opR (-\Delta)^\theta w_{q+1} } \ls \linfp{w_{q+1}}^{1-\theta_*} \linfp{\nabla w_{q+1}}^{\theta_*} \lesssim \ell^{-4} \lamqp^{\theta_*} r^{1-\frac{2}{p}}, \\
	\linfp{ v_\ell \ootimes w_{q+1} + w_{q+1} \ootimes v_\ell} \lesssim \cNN{v_\ell}{1} \linfp{w_{q+1}} \lesssim \ell^{-1} r^{1-\frac{2}{p}},
\end{gather*}
where $\theta_*$ is defined by \eqref{3.58+3}.
Thus, 
\begin{equation} \label{3.56_Rl_Est}
	\linfp{\Rl} \lesssim \ell^{-2} \sigma \mu r^{2-\frac{2}{p}} + \ell^{-4} \lamqp^{\theta_*} r^{1-\frac{2}{p}}.
\end{equation}
{Also} by \eqref{3.46_w_cN} and \eqref{+.8_vl_CN}, one has
\begin{align}
	\cNN{\Rl}{1} \ls & \cNN{\pt \wqp + \pt \wqc}{1} + \cNN{w_{q+1}}{2} + \cNN{v_\ell}{1} \cNN{w_{q+1}}{1} \notag\\
	\ls & \ell^{-8} r \lamqp^2 + \ell^{-5} r \lamqp \notag \\
	\ls & \ell^{-8} r \lamqp^2.  \label{7.11+}
\end{align}

At last, we shall get the estimates for $\Ro$,
which is the main part in the convex integration scheme.
By the definition \eqref{3.30_wqp_def} of $\wqp$, 
and noting \eqref{3.34}, \eqref{3.21+}, {one has}
\begin{align*}
	\wqp \ootimes \wqp + \Rls = & \sumkk a_\dir(t,x) a_{\dir'}(t,x) \mbw_\dir(t,x) \ootimes \mbw_{\dir'}(t,x) + \Rls \\
	= & \sumkk a_\dir a_{\dir'} \mbp_{\neq 0} \big(\mbw_\dir \ootimes \mbw_{\dir'} \big) 
\end{align*}
and 
\begin{align*}
	& \divg \big( \sumkk a_\dir a_{\dir'} \mbp_{\neq 0} (\mbw_\dir \ootimes \mbw_{\dir'}) \big) \\
	= & \divg \big( \sumkk a_\dir a_{\dir'} \mbp_{\geq (\lamqp \sigma)/{2}} (\mbw_\dir \ootimes \mbw_{\dir'}) \big) \\
	= & \frac{1}{2} \sumkk \mbp_{\neq 0} \Big( \nabla(a_\dir a_{\dir'}) \cdot \mbp_{\geq (\lamqp \sigma)/{2}} (\mbw_\dir \ootimes \mbw_{\dir'} + \mbw_{\dir'} \ootimes \mbw_{\dir} )\Big) \\
	& + \frac{1}{2} \sumkk \mbp_{\neq 0} \Big( a_\dir a_{\dir'} \divg \mbp_{\geq (\lamqp \sigma)/{2}} (\mbw_\dir \ootimes \mbw_{\dir'} + \mbw_{\dir'} \ootimes \mbw_{\dir} )\Big) \\
	:= & \frac{1}{2} \sumkk \big( \mce{\dir,\dir',1} + \mce{\dir,\dir',2} \big),
\end{align*}
Among these terms, by Lemma \ref{Lem:3.11}, and noting \eqref{3.24_mbw_est},
\begin{align}
	\linfp{\opR \mce{\dir,\dir',1}} \lesssim & \Linfp{ |\nabla|^{-1} \mce{\dir,\dir',1}} \notag \\
	\lesssim & (\lamqp \sigma)^{-1} \cNN{a_\dir a_{\dir'}}{3} \Linfp{ \mbw_\dir \ootimes \mbw_{\dir'}} \notag \\
	\lesssim & \ell^{-8} (\lamqp \sigma)^{-1} \linftwop{\mbw_\dir} \linftwop{\mbw_{\dir'}} \lesssim \frac{\ell^{-8}}{\lamqp \sigma} r^{2-\frac{2}{p}}. \label{3.57_E1_Est}
\end{align}

Since we use stationary 2D flow instead of the Beltrami flow in 3D,
we shall use a process slightly different from the one in \cite{Buckmuster_Vicol} and \cite{Luo_Titi} 
to estimate $\mce{\dir,\dir',2}$, 
see also Lemma 4 of \cite{Choffrut_DeLellis_Szekelyhidi}.
Noting the definition of $b_\dir$ and $\psi_\dir$, \eqref{3.1_def_bpsi}, 
and that $\dir,\dir' \in \Lambda \subset \mbs^1$,
it is direct to check that
\begin{align*}
	&	\big( \dir^\perp \ootimes  \dir'^\perp +\dir'^\perp \ootimes \dir^\perp  \big) (\dir+\dir') \\
	= &  (\dir \cdot \dir' -1 ) (\dir+\dir') \\
	= &  (\dir^\perp \cdot \dir'^\perp -1) (\dir+\dir').
\end{align*}
Thus, 
\begin{align*}
	& \divg \big( b_\dir \ootimes b_{\dir'} + b_{\dir'} \ootimes b_\dir \big) \\
	= & \divg \big( b_\dir \otimes b_{\dir'} + b_{\dir'} \otimes b_\dir - b_\dir \cdot b_{\dir'} \id  \big) \\
	= & - i \lamqp \big( \dir^\perp \otimes  \dir'^\perp +\dir'^\perp \otimes \dir^\perp - \dir^\perp \cdot \dir'^\perp \id \big) (\dir+\dir') \rme^{i \lamqp (\dir+\dir')\cdot x} \\
	= & i \lamqp  (\dir+\dir') \rme^{i \lamqp (\dir+\dir')\cdot x} \\
	= & \nabla \big(  \lamqp^2 \psi_\dir \psi_{\dir'} \big),
\end{align*}
and 
\begin{align*}
	& \divg \big( \mbw_\dir \ootimes \mbw_{\dir'} +  \mbw_{\dir'} \ootimes \mbw_{\dir}  \big) \\
	= & \big( b_\dir \ootimes b_{\dir'} + b_{\dir'} \ootimes b_\dir \big) \nabla (\eta_\dir \eta_{\dir'}) - \eta_\dir \eta_{\dir'} \nabla \big(  \lamqp^2 \psi_\dir \psi_{\dir'}  \big).
\end{align*}

Then for $\mce{\dir,\dir',2}$, if $\dir+\dir' \neq 0$, due to \eqref{3.21},
\begin{align*}
	& a_\dir a_{\dir'} \mbp_{\geq (\lamqp \sigma)/{2}} \divg \Big( \mbw_\dir \ootimes \mbw_{\dir'} +  \mbw_{\dir'} \ootimes \mbw_{\dir}  \Big) \\
	= & a_\dir a_{\dir'} \mbp_{\geq (\lamqp \sigma)/{2}} \big( (b_\dir \ootimes b_{\dir'} + b_{\dir'} \ootimes b_\dir ) \nabla(\eta_\dir \eta_{\dir'}) \big) \\
	& - \nabla \Big((a_\dir a_{\dir'}) \mbp_{\geq (\lamqp \sigma)/{2}}  \big(\eta_\dir \eta_{\dir'} \big( \lamqp^2 \psi_\dir \psi_{\dir'} \big) \big)\Big) \\
	& + \nabla(a_\dir a_{\dir'}) \cdot \mbp_{\geq (\lamqp \sigma)/{2}}  \big(\eta_\dir \eta_{\dir'} \big( \lamqp^2 \psi_\dir \psi_{\dir'} \big) \big)\\
	& + a_\dir a_{\dir'} \mbp_{\geq (\lamqp \sigma)/{2}}  \big( \nabla(\eta_\dir \eta_{\dir'}) \big( \lamqp^2 \psi_\dir \psi_{\dir'} \big) \big)\\
	:= & \mce{\dir,\dir',2,1} + \mce{\dir,\dir',2,2} + \mce{\dir,\dir',2,3} + \mce{\dir,\dir',2,4}.
\end{align*}
Among these terms, $\mce{\dir,\dir',2,2}$ can be added to the $p^* \id$ term,
$\mce{\dir,\dir',2,3}$ can be estimated as $\mce{\dir,\dir',1}$.
Moreover, as \eqref{3.20+_eta_FrePro}, by the definitions \eqref{3.1_def_bpsi} and \eqref{3.17_eata_def},
for the case $\dir+\dir'\neq 0$, 
we can change the projector $\mbp_{\geq (\lambda \sigma)/{2}}$ in $\mce{\dir,\dir',2,1}$ and $\mce{\dir,\dir',2,4}$
into $\mbp_{\geq {\lamqp}/{10}}$.
Then using Lemma \ref{Lem:3.11} and noting \eqref{3.25_eta_est}, 
\begin{align*}
	& \linfp{\opR \mbp_{\neq 0} \mce{\dir,\dir',2,1}} \\
	\lesssim & \Linfp{ |\nabla|^{-1} \mbp_{\neq 0} \Big( a_\dir a_{\dir'} \mbp_{\geq \lamqp/10} \big( b_\dir \ootimes b_{\dir'} + b_{\dir'} \ootimes b_\dir \big) \nabla(\eta_\dir \eta_{\dir'}) \Big)} \\
	\lesssim & \lamqp^{-1} \cNN{a_\dir a_{\dir'}}{2} \| b_\dir \|_{L^\infty_{t,x}}\| b_{\dir'} \|_{L^\infty_{t,x}} \big( \linftwop{\eta_\dir} \linftwop{\nabla \eta_{\dir'}} + \linftwop{\nabla \eta_\dir} \linftwop{\eta_{\dir'}} \big) \\
	\lesssim & \ell^{-6} \sigma r^{3-\frac{2}{p}}.
\end{align*}
Similarly,
\[
	\Linfp{\opR \mbp_{\neq 0} \mce{\dir,\dir',2,4}} \lesssim \ell^{-6} \sigma r^{3-\frac{2}{p}}.
\]
Thus, for $\dir+\dir'\neq 0$,
\begin{equation} \label{3.58_E2_Est}
	\Linfp{\opR \mce{\dir,\dir',2}} \lesssim \Big( \frac{\ell^{-8}}{\lamqp \sigma} + \ell^{-6} \sigma r \Big) r^{2-\frac{2}{p}}.
\end{equation}

Next, for the case $\dir+\dir'=0$ namely, for $\mce{\dir,-\dir,2}$ with $\dir \in \Lambda$,
we have 
$$ \nabla \big( \lamqp^2 \psi_\dir \psi_{\dir'} \big) = 0, $$ 
and by \eqref{3.18_eta_est}
\begin{align*}
	& \divg \big( \mbw_\dir \ootimes \mbw_{-\dir} + \mbw_{-\dir} \ootimes \mbw_\dir \big) \\
	= & \big( b_\dir \ootimes b_{-\dir} + b_{-\dir} \ootimes b_\dir \big) \nabla (\eta_\dir \eta_{-\dir}) \\
	= & 2 (\dir^\perp \ootimes \dir^\perp) \nabla \eta_\dir^2 = \big( \id - 2 \dir\otimes \dir \big) \nabla \eta_\dir^2 \\
	= & \Big( \nabla \eta_\dir^2 - 2 \big( (\dir\cdot \nabla) \eta_\dir^2 \big) \dir \Big) \\
	= & \Big( \nabla\eta_\dir^2 \mp 2 \frac{1}{\mu} \dir \pt \eta_\dir^2 \Big) \quad \text{for}\ \dir \in \Lambda^\pm.
\end{align*}
Thus, for $\dir \in \Lambda^\pm$, 
\begin{align*}
	\mce{\dir,-\dir,2} = & \mbp_{\neq 0} \Big(  a_\dir^2 \mbp_{\geq (\lamqp \sigma)/2} \big( \nabla \eta_\dir^2 \mp 2 \frac{1}{\mu} \dir \pt \eta_\dir^2 \big) \Big) \\
	= & \nabla \big( a_\dir^2 \mbp_{\geq (\lamqp \sigma)/2} \eta_\dir^2 \big) - \mbp_{\neq 0} \Big( (\nabla a_\dir^2) \mbp_{\geq (\lamqp \sigma)/2} \eta_\dir^2 \Big) \\
	& \mp 2 \mu^{-1} \dir \pt \mbp_{\neq 0} \big( a_\dir^2 \mbp_{\geq (\lamqp \sigma)/2} \eta_\dir^2 \big) \pm 2 \mu^{-1} \dir \mbp_{\neq 0} \big( (\pt a_\dir^2) \mbp_{\geq (\lamqp \sigma)/2} \eta_\dir^2 \big){.}
\end{align*}
Noting the definition of $\wqt$, \eqref{3.32_wqt_def}, \eqref{3.20+_eta_FrePro}, and that
\[
	\id - \mbp_H = \nabla \Delta^{-1} \nabla\cdot,
\]
one has
\begin{align*}
	& \frac{1}{2} \sumk \mce{\dir,-\dir,2} + \pt \wqt \\
	= & \Big(- \sumk \mu^{-1} \dir \nabla \Delta^{-1} \divg \pt \mbp_{\neq 0} \big( a_\dir^2 \mbp_{\geq (\lamqp \sigma)/2} \eta_\dir^2 \big) 
	+ \frac{1}{2}\nabla \big(a_\dir^2 \mbp_{\geq (\lamqp \sigma)/2} \eta_\dir^2\big) \Big) \\
	& + \Big( - \frac{1}{2} \sumk  \mbp_{\neq 0} \big(  \nabla a_\dir^2  \mbp_{\geq (\lamqp \sigma)/2}  \eta_\dir^2 \big) 
	\pm  \sumk  \mu^{-1} \dir \mbp_{\neq 0} \big(  \pt a_\dir^2 \mbp_{\geq (\lamqp \sigma)/2}  \eta_\dir^2 \big) \Big) \\
	:= & \mce{\dir,-\dir,2,1} + \mce{\dir,-\dir,2,2}.
\end{align*}
Here, $\mce{\dir,-\dir,2,1}$ can be added to the pressure term, 
and $\mce{\dir,-\dir,2,2}$ can be estimated with Lemma \ref{Lem:3.11} as
\begin{align}
	\Linfp{ \opR \mce{\dir,-\dir,2,2}} \lesssim & \frac{\ell^{-8}}{\lamqp \sigma}  \linftwopsq{\eta_\dir} \notag \\
	\lesssim & \frac{\ell^{-8}}{\lamqp \sigma} r^{2-\frac{2}{p}}. \label{3.66+}
\end{align}
Thus, combining \eqref{3.57_E1_Est}, \eqref{3.58_E2_Est} and \eqref{3.66+} yields,
\begin{equation} \label{3.63}
	\Linfp{\opR \divg ( \Ro -p^* \id ) } \lesssim \big( \frac{\ell^{-8}}{\lamqp \sigma} + \ell^{-6} \sigma r \big) r^{2-\frac{2}{p}},
\end{equation}
for
\begin{align*}
	p^* = & - \sum_{ \substack{\dir,\dir' \in \Lambda\\ \dir + \dir' \neq 0}} (a_\dir a_{\dir'}) \mbp_{\geq (\lamqp \sigma)/{2}}  \big(\eta_\dir \eta_{\dir'} \big( \lamqp^2 \psi_\dir \psi_{\dir'} \big) \big) \\
	& - \sumk \mu^{-1} \dir \Delta^{-1} \divg \pt \mbp_{\neq 0} \big( a_\dir^2 \mbp_{\geq (\lamqp \sigma)/2} \eta_\dir^2 \big) 
	+ \frac{1}{2} \big(a_\dir^2 \mbp_{\geq (\lamqp \sigma)/2} \eta_\dir^2\big).
\end{align*}
{Also} by \eqref{3.39}, \eqref{3.25_eta_est},
\begin{align}
	& \CNN{\opR \divg ( \Ro -p^* \id ) }{1} \notag \\
	\ls & \sumkk \CNN{\mce{\dir,\dir',1}}{2} + \sum_{ \substack{\dir,\dir' \in \Lambda\\ \dir + \dir' \neq 0}} \big(
	\CNN{\mce{\dir,\dir',2,1}}{2} + \CNN{\mce{\dir,\dir',2,3}}{2} + \CNN{\mce{\dir,\dir',2,4}}{2} \big) \notag \\
	& \quad + \sumk \CNN{\mce{\dir,-\dir,2,2}}{2} \notag \\
	\ls & 	\cNN{a_\dir}{3} \cNN{a_\dir}{0} (\cNN{\nabla^3 \eta_\dir}{0} + \cNN{\nabla^2 \pt \eta_\dir}{0}) \cNN{\eta_\dir}{0} \notag \\
	& \qquad \cdot \big( \cNN{b_\dir}{2} \cNN{b_\dir}{0} + \lamqp^3 \cNN{\psi_\dir}{2} \cNN{\psi_\dir}{0} \big) \notag \\
	\ls & \ell^{-8} \lamqp^5 \sigma^3 r^4 \mu. \label{7.15+}
\end{align}

Summing up \eqref{3.58+2}--\eqref{7.9+}, \eqref{3.56_Rl_Est}--\eqref{7.11+} and \eqref{3.63}--\eqref{7.15+},
we can get
\begin{align}
	\linfp{\oR_{q+1}} \lesssim & \ell^{-8} \Big( \sigma \mu + \sigma r + \mu^{-1} r + (\lamqp \sigma)^{-1} \Big) r^{2-\frac{2}{p}} + \ell^{-4} \lamqp^{\theta_*}  r^{1-\frac{2}{p}}, \label{7.16} \\
	\CNN{\oR_{q+1} }{1} \ls &  \ell^{-12} r^2 \lamqp^2 + \ell^{-8} \lamqp^5 \sigma^3 r^4 \mu. \label{7.17}
\end{align}

At last, we choose the parameters specifically as
\begin{equation} \label{3.64_parameter_choice}
	r = \lamqp^{1 - 6 \alpha},  \quad \mu = \lamqp^{1 - 4\alpha}, \quad \sigma = \lamqp^{-(1-2\alpha)},
\end{equation}
with $\alpha \in \mbq^+$ defined in \eqref{7.16+_alpha},
and choose $1 < p < 2$ such that
\[
	(1-6\alpha) (2-\frac{2}{p}) = \alpha,
\]
namely,
\[
	p = \frac{2-12\alpha}{2-13\alpha} \in (1,2) \quad \text{and} \quad r^{2-\frac{2}{p}} = \lamqp^\alpha.
\]
Then we can check that $r, \sigma, \mu$ satisfy the requirements in \eqref{3.19+_parameters}.
By choosing $A \in 5 \mbn$ large enough,
one can get \eqref{2.3_b_Rq_l1_est} and \eqref{2.3_c_RC1_est}.
And \eqref{3.46_w_cN} yields \eqref{2.3_a_vqC1_est}.
Meanwhile, by \eqref{3.39+} and \eqref{3.58+1},
we can get \eqref{2.4_suppv};
by \eqref{3.42_wqp_inf2}--\eqref{3.44_wqcwqt_infp},
we can get \eqref{2.5_L2Increase};
and by \eqref{+.11} and \eqref{3.43_w_infp},
we can get \eqref{2.6_WIncrease},
which completes the proof of Lemma \ref{Lem:2.1}.

\section*{Acknowledgments}
The authors would like to thank Professor Zhouping Xin for his encouragement and supports.

\appendix

\section{Geometric Lemma}
In this part, we shall give an elementary proof to Lemma \ref{Lem:3.2}.
In fact, if we set
\[
	\Gamma(s) = \left\{\begin{aligned}
		s+1, & \quad s > 0, \\
		1, & \quad s \leq 0,
	\end{aligned}\right.
\]
and 
\[
	\Gamma_* (s) = \Gamma * \tilde\varphi (s)
\]
with an even function $\tilde \varphi$ satisfying
\[
	\tilde{\varphi} \geq 0, \quad  \tilde{\varphi} \in C^\infty(\mbr), \quad \supp \tilde{\varphi} \subseteq (-1,1).
\]
Then we have 
\begin{gather*}
	\Gamma^* \in C^\infty(\mbr), \\
	1 \leq \Gamma_* \leq s+2, \quad \forall\, s \in \mbr.
\end{gather*}
And since
\[
	\Gamma(s) \cdot 1 + \Gamma(-s) \cdot (-1) = s, \quad \forall\, s \in \mbr,
\]
we have
\[
	\Gamma_*(s) \cdot 1 + \Gamma_*(-s) \cdot (-1) = s, \quad \forall\, s \in \mbr.
\]
Thus, for each
\[
	\oR = \begin{pmatrix}
		\oR_{11} & \oR_{12} \\
		\oR_{12} & -\oR_{11}
	\end{pmatrix}
\]
it is direct to check
\begin{align*}
	-\oR = &  \big( \frac{25}{14} \Gamma_*(-\oR_{11}) + \frac{25}{48} \Gamma_*(\oR_{12}) \big) \big( \dir_1^\perp \ootimes \dir_1^\perp + \dir_{-1}^\perp \ootimes \dir_{-1}^\perp \big) \\
	& + \big( \frac{25}{14} \Gamma_*(-\oR_{11}) + \frac{25}{48} \Gamma_*(-\oR_{12}) \big) \big( \dir_2^\perp \ootimes \dir_2^\perp + \dir_{-2}^\perp \ootimes \dir_{-2}^\perp \big) \\
	& + \big( \frac{25}{14} \Gamma_*(\oR_{11}) + \frac{25}{48} \Gamma_*(\oR_{12}) \big) \big( \dir_3^\perp \ootimes \dir_3^\perp + \dir_{-3}^\perp \ootimes \dir_{-3}^\perp \big) \\
	& + \big( \frac{25}{14} \Gamma_*(\oR_{11}) + \frac{25}{48} \Gamma_*(-\oR_{12}) \big) \big( \dir_4^\perp \ootimes \dir_4^\perp + \dir_{-4}^\perp \ootimes \dir_{-4}^\perp \big),
\end{align*}
where
\begin{gather*}
	\dir_1 = \frac{1}{5}(3e_1 + 4e_2), \; \dir_2 = \frac{1}{5}(3e_1 - 4e_2), \; \dir_3 = \frac{1}{5}(4e_1 + 3e_2), \; \dir_4 = \frac{1}{5}(4e_1 - 3e_2) \in \Lambda^+, \\
	\dir_{-1} = \frac{1}{5}(-3e_1 - 4e_2), \; \dir_{-2} = \frac{1}{5}(-3e_1 + 4e_2), \; \dir_{-3} = \frac{1}{5}(-4e_1 - 3e_2), \; \dir_{-4} = \frac{1}{5}(-4e_1 + 3e_2) \in \Lambda^-. 
\end{gather*}
And we can choose our smooth functions as follows
\begin{align*}
	& \gamma_{\dir_1}(\oR) = \gamma_{\dir_{-1}}(\oR) = \sqrt{\frac{25}{14} \Gamma_*(-\oR_{11}) + \frac{25}{48} \Gamma_*(\oR_{12})}, \\
	& \gamma_{\dir_2}(\oR) = \gamma_{\dir_{-2}}(\oR) = \sqrt{\frac{25}{14} \Gamma_*(-\oR_{11}) + \frac{25}{48} \Gamma_*(-\oR_{12})}, \\
	& \gamma_{\dir_3}(\oR) = \gamma_{\dir_{-3}}(\oR) = \sqrt{\frac{25}{14} \Gamma_*(\oR_{11}) + \frac{25}{48} \Gamma_*(\oR_{12})}, \\
	& \gamma_{\dir_4}(\oR) = \gamma_{\dir_{-4}}(\oR) = \sqrt{\frac{25}{14} \Gamma_*(\oR_{11}) + \frac{25}{48} \Gamma_*(-\oR_{12})}.
\end{align*}
Moreover, by the bounds of $\Gamma_*$, it is obvious that \eqref{3.14+} holds.



\end{document}